\documentclass[letterpaper,10pt,conference]{ieeeconf}

\usepackage{mathtools}
\usepackage{amssymb, amsmath, bm}
\usepackage[thmmarks, amsmath]{ntheorem}
\usepackage{xcolor}
\usepackage{balance} 

\usepackage[labelformat=simple]{subcaption}
\usepackage{hyperref}
\let\labelindent\relax
\usepackage{enumitem}
\newcommand\numberthis{\addtocounter{equation}{1}\tag{\theequation}}

\usepackage[style=ieee,dashed=false]{biblatex}
\IEEEoverridecommandlockouts                             
\title{\fontsize{15}{18}\selectfont \bf
System Identification from Partial Observations under Adversarial Attacks}

\author{Jihun Kim and Javad Lavaei
\thanks{This work was supported by the U. S. Army Research Laboratory and the U. S. Army Research Office under Grant W911NF2010219, Office of Naval Research under Grant N000142412673, AFOSR, NSF, and the UC Noyce Initiative. Jihun Kim and Javad Lavaei are with the Department of Industrial Engineering and Operations Research, University of California, Berkeley. 
Emails:  {\tt\footnotesize \{jihun.kim, lavaei\}@berkeley.edu}}
}

\theoremstyle{plain} 
\theoremnumbering{arabic} 
\theoremseparator{.} 
\newtheorem{theorem}{Theorem}
\newtheorem{lemma}{Lemma}

\theoremstyle{plain} 
\theoremheaderfont{\itshape} 
\theorembodyfont{\itshape} 
\theoremnumbering{arabic} 
\theoremseparator{:} 
\newtheorem{definition}{Definition}
\newtheorem{assumption}{Assumption}

\theoremstyle{plain}
\theoremheaderfont{\itshape}  
\theorembodyfont{\normalfont} 
\theoremseparator{:}          
\newtheorem{example}{Example}
\newtheorem{remark}{Remark}

\DeclareMathOperator*{\argmin}{arg\,min}
\allowdisplaybreaks

\begin{document}

\maketitle 
\thispagestyle{empty}
\pagestyle{empty}

\begin{abstract}
This paper is concerned with the partially observed linear system identification, where the goal is to obtain reasonably accurate estimation of the balanced truncation of the true system up to order $k$ from output measurements. We consider the challenging case of system identification under adversarial attacks, where the probability of having an attack at each time is $\Theta(1/k)$ while the value of the attack is arbitrary. We first show that the $\ell_1$-norm estimator exactly identifies the true Markov parameter matrix for nilpotent systems under any type of attack. We then build on this result to extend it to general systems and show that the estimation error exponentially decays as $k$ grows. The estimated balanced truncation model accordingly shows an exponentially decaying error for the identification of the true system up to a similarity transformation. This work is the first to provide the input-output analysis of the system with partial observations under arbitrary attacks. 
\end{abstract}

\section{Introduction}\label{intro}
Dynamical systems are often highly complex to accurately model from physics, which potentially leads to a considerable number of unknown parameters of the underlying system. The system identification  is to identify these true parameters, given the input and output data \cite{ljung1998sys}. In the fully observed system, all states are measured, meaning that the outputs are identical to the states. The challenge of system identification is often posed by the disturbances injected into the system. Existing methods to deal with this problem include least-squares \cite{simchowitz2018learningwithout,fara2018unstable,  jedra2020finite}, $\ell_2$-norm estimator \cite{yalcin2023exact, zhang2024exact}, and $\ell_1$-norm estimator \cite{kim2025acc}, where each estimator tackles a different type of disturbance.
While the classical least-squares method overcomes sub-Gaussian zero-mean independent disturbances, the work \cite{kim2025acc} considers the general case where the system is affected by sub-Gaussian, nonzero-mean, and possibly adversarial attacks. 

However, one may not be able to measure all states of the system in many applications, including robotics \cite{lauri2022survey}, healthcare \cite{ala2014heart}, and complex safety-critical systems \cite{ben2009sto}. 
This partial measurement of the states hinders accurate system identification since it introduces an additional challenge of inferring unmeasured states from the observations. 
For this reason, instead of directly estimating the system parameters, it would be beneficial to first estimate the Markov parameter matrix using the observations, since a sufficiently large set of Markov parameters enables accurate reconstruction of the original system \cite{skelton1994lqg, fledder2002markov}.

The existing literature mainly used the least-squares method to estimate the Markov parameters, assuming zero-mean, independent, sub-Gaussian disturbances \cite{oymak2022revisit, sarkar2021finite}. 
A variant of least-squares method is given in \cite{simchowitz2019semi}, 
where the disturbances are predictable based on past observations. 
While the least-squares method provides a satisfactory estimator for such restrictive disturbances, little is known about the partially observed system identification when the disturbances are fully selected adversarially, leveraging past information to enhance their adversarial nature.


\begin{figure}
    \centering
    \includegraphics[width=0.9\linewidth, height=0.46\linewidth]{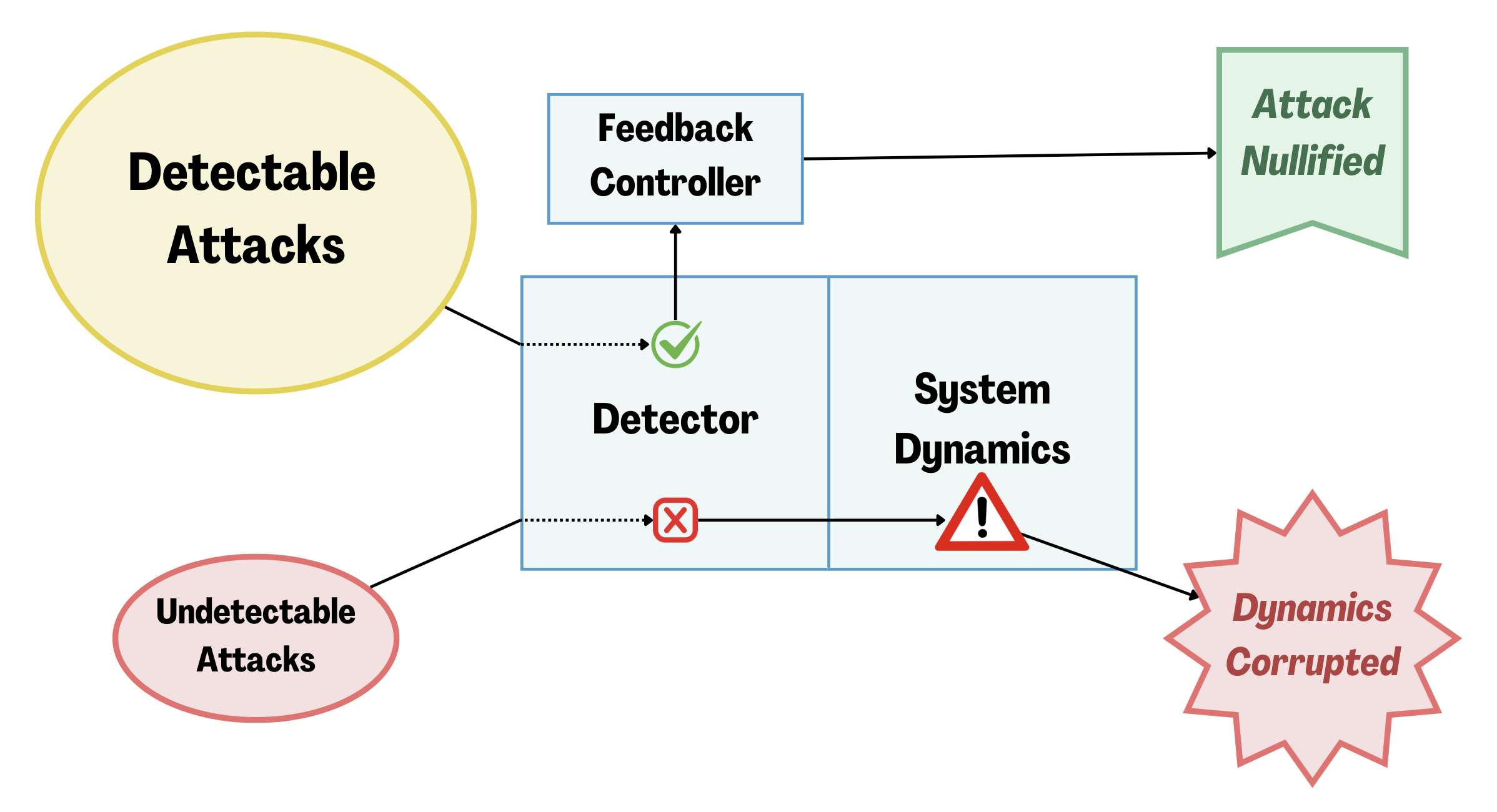}
    \caption{Detectable attacks vs. Occasional undetectable attacks}
    \label{fig:cyber}
    \vspace{-3mm}
\end{figure}
In this paper, we focus on the partially observed linear system identification and obtain the balanced truncated model of the true system up to order $k$, where we allow fully adversarial attacks to occur at each time with probability $\Theta(1/k)$.
Our attack model applies to the case when an extremely large attack may occasionally affect the system, such as natural disasters in power grids \cite{wang2016power, waseem2020partial}, unanticipated malicious cyberattacks \cite{pasqualetti2013cyber,hauser2017cyber}, and others. In particular, 
cyberattacks can be broadly classified as detectable or undetectable, where most attacks are effectively detected and nullified by well-designed detectors and controllers, preventing corruption of the system dynamics. However, undetectable attacks can occasionally occur when a strong adversary leverages complete knowledge of the system to craft a sophisticated attack \cite{mo2015secure, show2016system}. These so-called \textit{stealthy} attacks slip through the system while remaining unnoticed (see Figure \ref{fig:cyber}). 
    
We first estimate the Markov parameter matrix with an $\ell_1$-norm estimator by building on  \cite{kim2025acc}. We construct two scenarios on the true system and show that:
\begin{enumerate}
    \item the true Markov parameter matrix is the unique solution to the $\ell_1$-norm estimator for a nilpotent system,
    \item the estimation error of the Markov parameter matrix exponentially decays with $k$ for a general system.
\end{enumerate}

Following the estimation of the Markov parameter matrix, we retrieve the estimated balanced truncation up to order $k$, where we show that the error also decays exponentially with $k$ within the similarity transformation.

The paper is organized as follows.  Sections \ref{Prelim} and \ref{sec: prob formulation} present the preliminaries and problem formulation. In Section \ref{exact recovery}, we prove that the $\ell_1$-norm estimator achieves exact recovery for a nilpotent system and derive a bounded estimation error for a general system under adversarial attacks. 
Section \ref{retrievesys} leverages these results to obtain an accurate approximation of the true system. 
In Section \ref{sec: experiments}, we present numerical experiments to support our theoretical findings.
Finally, concluding remarks are provided in Section \ref{sec: conclusion}.

\textbf{Notation.} Let $\mathbb{R}^n$ denote the set of $n$-dimensional vectors and $\mathbb{R}^{n\times n}$ denote the set of $n\times n$ matrices. For a matrix $A$, $\|A\|_2$ denotes the spectral norm and $\|A\|_F$ denotes the Frobenius norm of the matrix. 
$A_{[n_1:n_2], [m_1:m_2]}$ denotes the submatrix of $A$ that contains the rows from the $n_1^\text{th}$ to the $n_2^\text{th}$ row and the columns from the $m_1^\text{th}$ to the $m_2^\text{th}$ column. Let $A^{-1}$ denote the inverse, $A^{\dagger}$ denote the pseudoinverse, and $A^T$ denote the transpose of the matrix. Let $I_n$ denote the identity matrix in $\mathbb{R}^{n\times n}$.
For a vector $x$, $\|x\|_1$ denotes the $\ell_1$-norm  and $\|x\|_2$ denotes the $\ell_2$-norm of the vector. 
For a scalar $z$, $\text{sgn}(z)=1$ if $z>0$, $\text{sgn}(z)=-1$ if $z<0$, and $\text{sgn}(z)=0$ if $z=0$.
Let $\mathbb{E}$ denote the expectation operator.  
For an event $\mathcal{E}$, $\mathbb{P}(\mathcal{E})$ denotes the probability of the event, and the function
$\mathbb{I}\{E\}$ equals $1$ if $E$ occurs and $-1$ otherwise.
We use $\Theta(\cdot)$ for the big-$\Theta$ notation, and $\tilde{\Theta}(\cdot)$ for the big-$\Theta$ notation hiding logarithmic factors. 
Let $N(\mu,\Omega)$ denote the Gaussian distribution with mean $\mu$ and covariance $\Omega$.
Finally, let $\mathbb{S}^{n-1}$ denote the set $\{y\in\mathbb{R}^n : \|y\|_2=1\}$.

\section{Preliminaries}\label{Prelim}
In this work, we model each attack on the system as a sub-Gaussian vector variable (note that bounded attacks automatically satisfy our assumption). We begin by introducing the scalar variable defined in \cite{vershynin2025high}. 

\begin{definition}[sub-Gaussian scalar variables]\label{subgdef1}
    A random variable $w\in \mathbb{R}$ is called sub-Gaussian if there exists $c>0$ such that
    \begin{equation}\label{subG}
    \mathbb{E}\Bigr[\exp\Bigr(\frac{w^2}{c^2}\Bigr)\Bigr]\leq 2.
    \end{equation}
    Its sub-Gaussian norm is denoted by $\|w\|_{\psi_2}$ and defined as 
    \begin{equation}\label{norm}
       \|w\|_{\psi_2} = \inf\left\{c>0: \mathbb{E}\Bigr[\exp\Bigr(\frac{w^2}{c^2}\Bigr)\Bigr]\leq 2\right\}. 
    \end{equation}
\end{definition}

The $\psi_2$-norm satisfies the norm properties: positive definiteness, homogeneity, and the triangle inequality. Note that the following properties hold for a sub-Gaussian $w$:
\begin{subequations}
    \begin{align}
        &\mathbb{E}[|w|]\leq c_1  \|w\|_{\psi_2}, \label{firstprop}\\ 
        & \mathbb{P}(|w|\geq s) \leq 2\exp(-c_2 s^2/\|w\|_{\psi_2}^2) ,\quad \forall s\geq 0,\label{secondprop}
        \\& \mathbb{E}[e^{\lambda w}]\leq \exp (c_3 \lambda^2 \|w\|_{\psi_2}^2), ~~\forall \lambda\in\mathbb{R} ~~ \text{if }\mathbb{E}[w] = 0,\label{thirdprop}
    \end{align}  
\end{subequations}
    where $c_1,c_2, c_3$ are positive constants. For example,  if $w\sim N(0, \gamma^2 )$, we have $\|w\|_{\psi_2} = \Theta(\gamma)$ due to \eqref{thirdprop}. Moreover,  property \eqref{secondprop} splits into two inequalities if $\mathbb{E}[w]=0$:
    \begin{subequations}
    \begin{align}
      & \mathbb{P}(w\geq s) \leq \exp(-c_2 s^2/\|w\|_{\psi_2}^2) ,\quad \forall s\geq 0,\label{plus}
        \\&\mathbb{P}(w\leq -s) \leq \exp(-c_2 s^2/\|w\|_{\psi_2}^2) ,\quad \forall s\geq 0.\label{minus}
    \end{align}
    \end{subequations}

We introduce the following useful lemmas to analyze the sum of independent noncentral sub-Gaussians \cite{vershynin2025high}. 
    



\begin{lemma}[Centering lemma]\label{centering}
    If $w$ is a sub-Gaussian, then so is $w-\mathbb{E}[w]$ and there exists $C>0$ such that
    \begin{equation}
        \|w-\mathbb{E}[w]\|_{\psi_2} \leq 
        C\|w\|_{\psi_2}.
    \end{equation}
\end{lemma}
\begin{lemma}\label{independent}
    Let $w_1,\dots, w_N$ be independent, mean zero, sub-Gaussian random variables. Then, $\sum_{i=1}^N w_i$ is also sub-Gaussian and its sub-Gaussian norm is $\Theta\bigr((\sum_{i=1}^N \|w_i\|_{\psi_2}^2)^{1/2}\bigr)$. For example, if $w\sim N(0, \gamma^2 I_m)$, then  $\left\|\|w\|_2\right\|_{\psi_2} = \Theta(\gamma\sqrt{m})$ due to Jensen's inequality.
\end{lemma}

We introduce the notion of sub-Gaussian vectors below. 
\begin{definition}[sub-Gaussian vector variables]\label{subgvec}
    A random vector $w\in\mathbb{R}^d$ is called sub-Gaussian if for every $x\in\mathbb{R}^d$, $w^T x$ is a sub-Gaussian variable. Its norm is defined as
    \begin{equation}\label{normdef}
        \|w\|_{\psi_2} = \sup_{\|x\|_2 \leq 1, x\in\mathbb{R}^d}\|w^T x\|_{\psi_2}. 
    \end{equation}
\end{definition}
\medskip




\section{Problem Formulation} \label{sec: prob formulation}

Consider a linear time-invariant system represented by:
\begin{align*}
    x_{t+1} &= A x_t + Bu_t+w_t,\numberthis\label{sysd}\\y_t&= Cx_t + Du_t, \hspace{10mm}t=0,1,\dots,
\end{align*}
where $A\in\mathbb{R}^{n \times n}, B\in\mathbb{R}^{n \times m}, C\in\mathbb{R}^{r \times n}, D\in\mathbb{R}^{r \times m}$ are unknown system matrices, $x_t\in\mathbb{R}^n$ is the state, $u_t\in \mathbb{R}^m$ is the control input, and $y_t \in \mathbb{R}^r$ is the observation at time $t$.
$w_t\in \mathbb{R}^n$ is the attack injected into the system at time $t$ which occasionally happens. We assume that the attack times are selected with probability $p$, and
$w_t$ is identically zero when the system is not under attack. We allow $w_t$ to be completely arbitrarily chosen by an adversary at the attack times. 

We design the control inputs $u_0, u_1, \dots$ to be Gaussian. Given the observation trajectory $y_0, y_1,\dots$,
 our goal is to accurately approximate the true matrices $A, B, C,D$. We assume that $\|A\|_2$ is less than 1 and $x_0, w_0, w_1, \dots$ are all sub-Gaussian to prevent an unbounded growth of the system states.  
We formally present our assumptions below. 

\begin{assumption}[Spectral Norm]\label{stability}
    It holds that $\|A\|_2< 1$, \textit{i.e.}, the maximum singular value of $A$ is less than 1 (this condition can be relaxed to stability, as stated in Remark \ref{gotostab}).
\end{assumption}

\begin{assumption}[Sub-Gaussian norm]\label{maxnorm}
 Define a filtration $\mathcal{F}_t = \bm{\sigma}\{x_0, w_0, \dots, w_{t-1}\}$.
    There exists $\eta >0$ such that $\|x_0\|_{\psi_2} \leq \eta$ and 
    $\|w_t\|_{\psi_2}\leq \eta$ conditioned on $\mathcal{F}_t$ for all $t\geq 0$.
\end{assumption}

Under partial observability, the behavior of the system transferred from the control inputs $u_t, u_{t-1}, \dots, u_0$ to the output observation $y_t$ is represented by the transfer function $C(zI-A)^{-1}B+D$, a function involving the coefficients $CB$, $CAB$, $CA^2 B$, and so forth. Thus, it is generally impossible to characterize the observability without the interaction between $A$, $B$, and $C$. To this end, the Hankel matrix provides a tool for systemic input-output analysis. We introduce this notion below.

\begin{definition}[Hankel Matrix]\label{hankeldef}
    The $(\alpha, \beta)$-dimensional Hankel matrix for $(A,B,C)$ is defined as 
    \[
    \mathcal{H}_{\alpha,\beta}=  \begin{bmatrix}
        CA^\alpha B & CA^{\alpha+1}B & \cdots & CA^{\alpha+\beta-1}B \\ CA^{\alpha+1}B & CA^{\alpha+2}B & \cdots & CA^{\alpha+\beta}B &  \\ \vdots & \vdots & \ddots &\vdots \\ CA^{\alpha+\beta-1}B & CA^{\alpha+\beta}B & \cdots& CA^{\alpha+2\beta-2}B
    \end{bmatrix}. 
    \]
    We also denote $\mathcal{\bar{H}}_{\alpha,\beta}$ as the zero-padded matrix of $\mathcal{H}_{\alpha,\beta}$, where the right and bottom parts are extended infinitely with zeros,  with $\mathcal{H}_{\alpha,\beta}$ as its leading principal submatrix. 
\end{definition}

We aim to approximate the full Hankel matrix $\mathcal{H}_{0,\infty}$ given observations and control inputs. As a proxy of $\mathcal{H}_{0,\infty}$, 
we will estimate the Hankel matrix $\mathcal{H}_{0, k}$ for some natural number $k$, which requires the information of $CB, CAB, \dots, CA^{2k-2}B$. To this end, we define the following notion.
\begin{definition}[Markov parameter matrix]\label{markov}
    From the true system $(A,B,C,D)$, the Markov parameter matrix required to recover the matrix $D$ and the Hankel matrix $\mathcal{H}_{0,k}$ is denoted as $G_k^*$ and defined by 
    \begin{equation}
        G_k^* = [D~CB~ CAB~ \cdots ~ CA^{2k-2}B].
    \end{equation}
\end{definition}

To establish the relationship between the observations and the control inputs, one can write
\begin{subequations}\label{firstsecondthird}
    \begin{align*}
        y_t &= G_k^*\cdot [u_t^T~ u_{t-1}^T ~ \cdots~ u_{t-2k+1}^T]^T \numberthis\label{firstline}\\ &\hspace{2.5mm}+ [C~ CA ~ \cdots ~ CA^{2k-2}] \cdot[w_{t-1}^T~ \cdots ~ w_{t-2k+1}^T]^T\numberthis\label{secondline}\\&\hspace{2.5mm}+ CA^{2k-1}x_{t-2k+1}. \numberthis\label{thirdline}
    \end{align*}
\end{subequations}

Based on \eqref{firstsecondthird}, we propose the following $\ell_1$-norm estimator given $T$ observations $y_{2k-1},\dots, y_{T+2k-2}$ and the control inputs $u_0, \dots, u_{T+2k-2}$:
   \begin{equation}\label{l1est}
      \min_{G\in \mathbb{R}^{r\times 2km}} \sum_{t=2k-1}^{T+2k-2}\|y_t-G\mathbf{U}_t^{(k)}\|_1,
   \end{equation}
   where $\mathbf{U}_t^{(k)} = [u_t^T ~ u_{t-1}^T ~ \cdots ~ u_{t-2k+1}^T]^T$. 
    We will show that the $\ell_1$-norm estimator successfully overcomes adversarial attacks and that the estimate will be close to the true Markov parameter matrix $G_k^*$ within finite time.

However, solving for $A$, $B$, $C$ from $G_k^*$ is a nonconvex problem, resulting in infinitely many solutions up to a similarity transformation. To address this issue, it turns out that the balanced truncation can be recovered up to order $k$ from $G_k^*$.  We formally introduce this notion given in \cite{kung1981hankel} below.

\begin{definition}[$d$-order balanced truncated model]\label{dorder}
    Let the singular value decomposition (SVD) of the matrix $\mathcal{H}_{0,\infty}$ be given as $U\Sigma V^T$, where $\Sigma\in \mathbb{R}^{n\times n}$ is a diagonal matrix with singular values $\sigma_1\geq \dots\geq \sigma_n\geq 0$. Then for any $d\in \{1,\dots n\}$, the $d$-order balanced truncated model is defined as
    \begin{subequations}
         \begin{align*}
        &C^{(d)} = (U\Sigma^{1/2})_{[1:r], [1:d]}, ~ B^{(d)} = (\Sigma^{1/2}V^T)_{[1:d], [1:m]} \\ & A^{(d)} = (U\Sigma^{1/2})_{[1:\infty], [1:d]}^{\dagger} (U\Sigma^{1/2})_{[r+1:\infty], [1:d]}
    \end{align*}
    \end{subequations}
\end{definition}

Our ultimate goal is to recover a precise estimate of the balanced truncated model $A^{(d)}, B^{(d)}, C^{(d)}$ up to $d\in\{1,\dots,k\}$ given a predetermined $k$, under the accurate estimate of  $G_k^*$ obtained via the $\ell_1$-norm estimator. However, the occurrence of adversarial attacks potentially hinders the recovery of a high-order model. To potentially mitigate the impact of arbitrarily malicious attacks, 
we introduce an assumption on the attack time probability. 

\begin{assumption}[Probabilistic Attack]\label{null}
    $w_t$ is an attack at each time $t$ with probability $p<\frac{1}{4k-2}$ conditioned on $\mathcal{F}_t$, meaning that there exists a sequence $(\xi_t)_{t\geq 0}$ of independent $\text{Bernoulli}(p)$ variables, each independent of any $\mathcal{F}_t$ and 
    \begin{equation}\label{include}
    \{\xi_t = 0\} \subseteq \{w_t=0\}
    \end{equation}
    holds for all 
    $t\geq 0$.
\end{assumption}

Assumption \ref{null} implies that the system is not under attack at time $t$ if $\xi_t=0$. 
Note that $k$ can be selected \textit{independently} of the system order $n$, and thus the attack probability can be chosen without dependence on $n$.









\section{Estimation of the Markov Parameter Matrix with the $\ell_1$-norm Estimator}\label{exact recovery}
In this section, we will bound the estimation error of $G_k^*$ using the $\ell_1$-norm estimator.
Let $\hat{G}_k$ denote any estimate obtained from \eqref{l1est}. Equivalently, $\hat{G}_k$ can be expressed as a solution to
\begin{equation}\label{1no}
    \argmin_{G\in \mathbb{R}^{r\times 2km}} \sum_{t=2k-1}^{T+2k-2}\Bigr\|(G_k^*-G)\mathbf{U}_t^{(k)}+v_t + CA^{2k-1}x_{t-2k+1}\Bigr\|_1
\end{equation}
due to the equation given in \eqref{firstsecondthird}, where we use $v_t\in\mathbb{R}^r$ to denote the term given in \eqref{secondline}. 

\subsection{Exact Recovery for a Nilpotent System}\label{4a}

In this subsection, we first assume that $A$ is nilpotent with $A^{2k-1} = 0$; \textit{i.e.} the term in \eqref{thirdline} is zero. This will later be generalized to the case where $A^{2k-1}\neq 0$ in the next subsection. At this stage, we provide sufficient conditions under which $G_k^*$ is the unique solution to the $\ell_1$-norm estimator, assuming $A$ is nilpotent. For the following theorem, 
let $v_{t}^i$ denote the $i^\text{th}$ entry of $v_t$ for $i\in\{1,\dots,r\}$.


\begin{theorem}\label{optim}
    Suppose that $A^{2k-1}=0$. Then, $G_k^*$ is the unique solution to the $\ell_1$-norm estimator \eqref{l1est} if
    \begin{align}
            &\sum_{t=2k-1}^{T+2k-2} \mathbb{I}\{v_t^i = 0\}\cdot |s^T \mathbf{U}_t^{(k)}|> 0, \quad \forall s\in \mathbb{S}^{2km-1}\label{znotationa}
        \end{align}
  holds for all $i\in \{1,\dots,r\}$.
\end{theorem}
\begin{proof}
Since $A^{2k-1}=0$, an equivalent condition for $G_k^*$ to be the unique solution of the convex optimization problem \eqref{1no} is the existence of some $\epsilon >0$ such that
\begin{align}\label{diff1}
    \nonumber &\sum_{t=2k-1}^{T+2k-2} \|v_t\|_1 < \sum_{t=2k-1}^{T+2k-2} \|\Delta\cdot \mathbf{U}_t^{(k)} + v_t\|_1, \\&\hspace{30mm}\forall \Delta\in\mathbb{R}^{r\times 2km}: 0<\|\Delta\|_F \leq \epsilon,
\end{align}
  since a strict local minimum in convex problems implies the unique global minimum. 
    A sufficient condition for \eqref{diff1} is to satisfy all coordinate-wise inequalities. That is, if there exist $\epsilon_1, \dots, \epsilon_r >0$ such that
    \begin{align}\label{diff2}
    \nonumber&\sum_{t=2k-1}^{T+2k-2} |v_t^i| < \sum_{t=2k-1}^{T+2k-2} |\Delta_i^T \mathbf{U}_t^{(k)} + v_t^i|, \\&\hspace{30mm} \forall \Delta_i\in\mathbb{R}^{2km} : 0<\|\Delta_i\|_2 \leq \epsilon_i
\end{align}
for all $i\in\{1,\dots,r\}$, then the inequality \eqref{diff1} is satisfied. 
Note that we have
\begin{align*}
    |\Delta_i^T \mathbf{U}_t^{(k)} + v_t^i| -|v_t^i|\geq -|\Delta_i^T \mathbf{U}_t^{(k)}|  \numberthis\label{maybe}
\end{align*}
due to the triangle inequality. Substituting \eqref{maybe} into  \eqref{diff2} for $v_t^i \neq 0$ provides the sufficient conditions for \eqref{diff2}:
\begin{align}\label{simplify}
\sum_{t=2k-1}^{T+2k-2}   \mathbb{I}\{v_t^i=0\}\cdot |\Delta_i^T \mathbf{U}_t^{(k)}|>0
\end{align}
for all $0<\|\Delta_i\|_2 \leq \epsilon_i$. 
For every $i$, 
dividing both sides by $\|\Delta_i\|_2 >0$ leads  to
the set of inequalities in \eqref{znotationa}.
\end{proof}


To ensure that $G_k^*$ is the only solution for the $\ell_1$-norm estimator,
it suffices to show that the random variables on the left-hand side of \eqref{znotationa} are sufficiently positive. 
Before we provide the main theorem, the following lemma is useful.
\begin{lemma}\label{lemmafixed}
    Suppose that Assumption \ref{null} holds and $u_t\sim N(0, \gamma^2 I_m)$ for all $t$. For a fixed $s\in\mathbb{S}^{2km-1}$, we have
    \begin{align}\label{ntnttk}
    \nonumber&\mathbb{P}\biggr(\sum_{t=2k-1}^{T+2k-2} \mathbb{I}\{v_t^i = 0\}\cdot |s^T \mathbf{U}_t^{(k)}| \geq \Theta(\gamma (1-2q)T)\biggr) \\&\hspace{25mm}\geq 1-\exp\Bigr(-\Theta\Bigr(\frac{T(1-2q)^2}{k}\Bigr)\Bigr),
    \end{align}
    where $q:=1-(1-p)^{2k-1}.$
\end{lemma}
\begin{proof}
    The proof is provided in Appendix \ref{lemma3proof}.
\end{proof}

Lemma \ref{lemmafixed} presents a lower bound for a fixed $s$.
Note that $q<0.5$ holds  under Assumption \ref{null} since $p<\frac{1}{4k-2}$, which implies that the lower bound $\Theta(\gamma (1-2q)T)$ is positive.
 To ensure that the same order of the  lower bound uniformly holds for all $s\in\mathbb{S}^{2km-1}$, the following lemma analyzes the difference in the quantity evaluated at two different points.  

\begin{lemma}\label{lemmadiff}
 Suppose that $u_t\sim N(0, \gamma^2 I_m)$ for all $t$.    Given $\delta\in(0,1]$, when $T\geq \Theta(\frac{1}{m}\log(\frac{1}{\delta}))$, the inequality
   \begin{align*}
   &\sum_{t=2k-1}^{T+2k-2}\mathbb{I}\{v_t^i = 0\}\cdot |s^T \mathbf{U}_t^{(k)}| - \sum_{t=2k-1}^{T+2k-2}\mathbb{I}\{v_t^i = 0\}\cdot |\tilde s^T \mathbf{U}_t^{(k)}| \\&\hspace{40mm}\geq -\Theta(T\|s-\tilde s\|_2 \cdot  \gamma \sqrt{km})
   \end{align*}
   holds for all $s,\tilde s\in\mathbb{S}^{2km-1}$ with probability at least $1-\frac{\delta}{2}$.
\end{lemma}

\begin{proof}
    The proof can be found in Appendix \ref{lemma4proof}.
\end{proof}


The following theorem proves that the sufficient condition \eqref{znotationa} is indeed satisfied even in the presence of arbitrary and possibly adversarial attacks. 

\begin{theorem}\label{exactrecov}
    Suppose that Assumption \ref{null} holds. Let $u_t\sim N(0, \gamma^2 I_m)$ for all $t$.  Given $\delta\in(0,1]$, when 
    \begin{equation}\label{realtimebound}
        T\geq \Theta\left(\frac{k}{(1-2q)^2}\left[km\log\Bigr(\frac{km}{1-2q}\Bigr)+\log\Bigr(\frac{r}{\delta}\Bigr)\right]\right),
    \end{equation}
    we have 
    \begin{align}
        &\nonumber\sum_{t=2k-1}^{T+2k-2}\mathbb{I}\{v_t^i = 0\}\cdot |s^T \mathbf{U}_t^{(k)}| \geq \Theta(\gamma (1-2q)T)>0, \\&\hspace{33mm}\forall s\in\mathbb{S}^{2km-1}, \forall i\in\{1,\dots,r\}
    \end{align}
    with probability at least $1-\delta$, where  $q=1-(1-p)^{2k-1}$. 
\end{theorem}

\begin{proof}
By Lemma \ref{lemmafixed}, when
\begin{equation}\label{timebound}
    T\geq \Theta\Bigr(\frac{k}{(1-2q)^2}\log\Bigr(\frac{1}{\delta}\Bigr)\Bigr),
\end{equation}
the inequality
\begin{equation}\label{firsteqq}
\sum_{t=2k-1}^{T+2k-2}\mathbb{I}\{v_t^i = 0\}\cdot |s^T \mathbf{U}_t^{(k)}|\geq\Theta(\gamma (1-2q)T) >0
\end{equation}
holds for a fixed $s$ with probability at least $1-\frac{\delta}{2}$.

To obtain a positive lower bound on $\mathbb{I}\{v_t^i = 0\}\cdot |s^T \mathbf{U}_t^{(k)}|$ for all $s\in\mathbb{S}^{2km-1}$, we use a lemma from \cite{vershynin2025high} stating that one can select an $\epsilon$-net $\mathcal{N}_\epsilon$ consisting of 
 $(1+\frac{2}{\epsilon})^{2km}$  points such that for every $\tilde{s}\in\mathbb{S}^{2km-1}$, there exists $s\in\mathcal{N}_\epsilon$ satisfying $\|s-\tilde s\|\leq \epsilon$.
 
 We use $\epsilon^* = \Theta(\frac{1-2q}{\sqrt{km}})$. From Lemma \ref{lemmadiff}, for all $s, \tilde s \in \mathbb{S}^{2km-1}$ satisfying $\|s-\tilde s\|_2 \leq \epsilon^*$, we have 
 \begin{align}\label{secondeqq}
 \nonumber &\sum_{t=2k-1}^{T+2k-2}\mathbb{I}\{v_t^i = 0\}\cdot |s^T \mathbf{U}_t^{(k)}| - \sum_{t=2k-1}^{T+2k-2} \mathbb{I}\{v_t^i = 0\}\cdot |\tilde s^T \mathbf{U}_t^{(k)}| \\&\hspace{40mm}\geq -\frac{1}{2}\Theta(\gamma(1-2q)T).
 \end{align}
 with probability at least $1-\frac{\delta}{2}$, with the time \eqref{timebound}.
Considering \eqref{firsteqq} and \eqref{secondeqq}, it suffices to select $\Theta((1+\frac{2\sqrt{km}}{1-2q})^{2km})$ points  $s$ satisfying \eqref{firsteqq} with probability at least $1-\frac{\delta}{2\cdot \Theta((1+\frac{2\sqrt{km}}{1-2q})^{2km})}$ to guarantee that 
\begin{align}\label{conclusion}
    &\sum_{t=2k-1}^{T+2k-2}\mathbb{I}\{v_t^i = 0\}\cdot |s^T \mathbf{U}_t^{(k)}|\geq\frac{1}{2}\Theta(\gamma (1-2q)T) , 
\end{align}
for all $s\in\mathbb{S}^{2km-1}$
holds with probability at least $1-\delta$. Thus, we replace $\delta$ in \eqref{timebound} with $\frac{\delta}{\Theta((1+\frac{2\sqrt{km}}{1-2q})^{2km})}$ to arrive at 
\begin{equation}\label{conclusion2}
    T\geq  \Theta\biggr(\frac{k}{(1-2q)^2}\biggr[km\log\Bigr(\frac{km}{1-2q}\Bigr)+ \log\Bigr(\frac{1}{\delta}\Bigr)\biggr]\biggr).
\end{equation}
 Finally, to satisfy \eqref{conclusion} for all $i\in\{1,\dots,r\}$, we substitute $\frac{\delta}{r}$ for $\delta$ in \eqref{conclusion2} to obtain \eqref{realtimebound}.
\end{proof}

\begin{remark}
Theorem \ref{exactrecov} implies that \eqref{znotationa} indeed holds, ensuring that $G_k^*$ is the only solution to the $\ell_1$-norm estimator under the assumption that $A^{2k-1} = 0$. Each attack can be chosen in a fully adversarial manner without any assumption on its expectation. The exact recovery of $G_k^*$ is guaranteed when the attack probability satisfies $p<\frac{1}{4k-2}$, which represents a scenario where attacks of large magnitude may occasionally occur in arbitrary directions. 

\end{remark}

\subsection{Estimation Error for a General System}\label{4c}

In the previous subsection, we have discussed that under
the assumption $A^{2k-1} = 0$, the estimation error to obtain the true Markov parameter matrix $G_k^*$ is exactly zero after finite time.  However, in general, this exact recovery cannot be achieved since the term $CA^{2k-1}x_{t-2k+1}$ in \eqref{thirdline} remains nonzero at all times, given that exponential decay does not cause the term to vanish to zero.
 In this section, we derive an estimation error bound  when $A^{2k-1}\neq 0$. It turns out that the error is proportional to $\|A^{2k-1}\|_2$ (since it is assumed that $\|A\|_2<1$, this exponential term is expected to be small).
Before presenting the main theorem, the following lemma is helpful to bound the sum of state norms.

\begin{lemma}\label{sumxt}
    Suppose that Assumptions \ref{stability} and \ref{maxnorm} hold. Let $u_t\sim N(0, \gamma^2 I_m)$ for all $t$. Given $\delta\in(0,1]$, when $T\geq \Theta(\log(\frac{1}{\delta }))$,
    \begin{equation}
        \sum_{t=0}^{T-1} \|x_t\|_2 \leq \Theta\Bigr(\frac{(\eta+\gamma \sqrt{m}\cdot \|B\|_2)T}{1-\|A\|_2}\Bigr)
    \end{equation}
    holds with probability at least $1-\delta$. 
\end{lemma}
\begin{proof}
    The proof details are given in Appendix \ref{lemma5proof}.
\end{proof}

\begin{theorem}\label{esterror}
Suppose that Assumptions \ref{stability}, \ref{maxnorm}, and \ref{null} hold. Let $u_t\sim N(0, \gamma^2 I_m)$ for all $t$.  Define $q:=1-(1-p)^{2k-1}<0.5$.  
Let $\hat G_k$ be any solution to the $\ell_1$-norm estimator \eqref{l1est} and $G_k^*$ be the true Markov parameter matrix. 
Given $\delta\in(0,1]$, after the finite time in \eqref{realtimebound}, we have
\begin{equation*}
    \|G_k^* - \hat{G}_k\|_F \leq \Theta\biggr(\frac{\|A^{2k-1}\|_2\|C\|_2\sqrt{r} }{(1-\|A\|_2)(1-2q)}\cdot\Bigr(\frac{\eta}{\gamma}+\sqrt{m}\|B\|_2\Bigr)\biggr)
\end{equation*}
with probability at least $1-\delta$.
\end{theorem}

\begin{proof}
After the finite time in \eqref{realtimebound}, 
 we have
\begin{align*}
    \sum_{t=2k-1}^{T+2k-2} |\Delta_i^T &\mathbf{U}_t^{(k)} + v_t^i| - |v_t^i|\geq \sum_{t=2k-1}^{T+2k-2} \mathbb{I}\{v_t^i = 0\} |\Delta_i^T \mathbf{U}_t^{(k)}| \\&=\|\Delta_i\|_2\sum_{t=2k-1}^{T+2k-2} \mathbb{I}\{v_t^i = 0\}\cdot \left|\Bigr(\frac{\Delta_i}{\|\Delta_i\|_2}\Bigr)^T \mathbf{U}_t^{(k)}\right|\\&\geq \|\Delta_i\|_2 \cdot \Theta(\gamma (1-2q)T) ,\numberthis\label{lobo}
\end{align*}
for every $i\in\{1,\dots,r\}$, where the first inequality is from  the relationship between \eqref{diff2} and \eqref{simplify}, and the last inequality leverages Theorem \ref{exactrecov}.

Meanwhile, note that the optimality of $\hat{G}_k$ in \eqref{1no} induces
\begin{align*}
&\sum_{t=2k-1}^{T+2k-2} \|(G_k^*-\hat G_k)\mathbf{U}_t^{(k)} + v_t\|_1 -\|CA^{2k-1}x_{t-2k+1}\|_1
\\&\hspace{7mm}\leq\sum_{t=2k-1}^{T+2k-2} \|(G_k^*-\hat G_k)\mathbf{U}_t^{(k)} + v_t +CA^{2k-1}x_{t-2k+1}\|_1 \\&\hspace{7mm}\leq \sum_{t=2k-1}^{T+2k-2} \| v_t +CA^{2k-1}x_{t-2k+1}\|_1 \\&\hspace{7mm}\leq \sum_{t=2k-1}^{T+2k-2} \| v_t\|_1 +\sum_{t=0}^{T-1}\|CA^{2k-1}x_{t}\|_1,
\end{align*}
where the first and third inequalities are due to the triangle inequality. 
For $i\in\{1,\dots,r\}$, let $g_i^*$ and $\hat g_i$ denote the $i^\text{th}$ rows of $G_k^*$ and $\hat G_k$, respectively.
Then, we have
\begin{align*}
\sum_{i=1}^r &f_i(g_i^* - \hat g_i) = 
\sum_{t=2k-1}^{T+2k-2} \|(G_k^*-\hat G_k)\mathbf{U}_t^{(k)} + v_t\|_1-\|v_t\|_1\\& \leq \sum_{t=0}^{T-1} 2\|CA^{2k-1}x_{t}\|_1\leq \sum_{t=0}^{T-1} 2\sqrt{r}\|CA^{2k-1}x_{t}\|_2,\numberthis\label{2r2norm}
\end{align*}
where the right-hand side is upper-bounded 
using Lemma \ref{sumxt} and the left-hand side is lower-bounded with \eqref{lobo}. Consequently, it follows from \eqref{2r2norm} that
\begin{align*}
    \sum_{i=1}^r \|g_i^* - &\hat g_i\|_2 \cdot \Theta(\gamma (1-2q)T)\\&\leq 2\sqrt{r} \|C\|_2 \|A^{2k-1}\|_2\Theta\Bigr(\frac{(\eta+\gamma \sqrt{m}\cdot \|B\|_2)T}{1-\|A\|_2}\Bigr).
\end{align*}
The relationship $\|G_k^*-\hat G_k\|_F \leq \sum_{i=1}^r \|g_i^* - \hat g_i\|_2$ completes the proof.
\end{proof}

\begin{remark}
    The estimation error bound in Theorem \ref{esterror} is $ \Theta(\|A^{2k-1}\|_2)$, which implies that small $\|A\|_2$ and large $k$ reduce the estimation error. 
    Large $k$ is beneficial in the sense that one can recover up to $A^{(k)}, B^{(k)}, C^{(k)}$ given in Definition \ref{dorder}.
    However, it is worth noting that the attack probability in Assumption \ref{null} is restricted to $p<\frac{1}{4k-2}$ to guarantee the proposed error. Thus, to ensure the practicality of our scenario,  $k$ cannot be chosen arbitrarily large for the recovery of $A^{(k)}, B^{(k)}, C^{(k)}$; rather, the desired degree of recovery should be specified in advance.
\end{remark}
\begin{remark}\label{gotostab}
    The term $\frac{\|A^{2k-1}\|_2}{1-\|A\|_2}$ in the error bound comes from Assumption \ref{stability} when bounding the quantity $\sum_{i=0}^\infty\|A^i\|_2$.
    The assumption can actually be relaxed to the general system stability assumption $\rho(A) < 1$, where $\rho(A)$ is the maximum absolute eigenvalue of $A$. This follows from Gelfand's formula, which establishes a finite upper bound of $\Phi(A)=\sup_{\tau\geq 0} \frac{\|A^\tau\|_2}{\rho(A)^{\tau}}$, which only depends on the system order $n$. In that case, the aforementioned error bound scales as  $\frac{\rho(A)^{2k-1}}{1-\rho(A)}$ multiplied by a factor depending only on $n$. 
\end{remark}

\section{Retrieving the True System from the \\Markov Parameter Matrix}\label{retrievesys}
In this section, we use an estimated Markov parameter matrix obtained in Section \ref{exact recovery} to recover $A, B, C, D$, which define the true system. In particular, we will provide an analysis on the $k$-order balanced truncation, where we leverage the result of the work \cite{sarkar2021finite}. Before presenting the theorem, we adopt the estimates of a $k$-order model based on the Ho-Kalman algorithm \cite{hokalman1966kalman}.

\begin{definition}[Estimates for $k$-order truncated model]
One can construct  $\mathcal{H}_{0,k}$ from $G^*_k$ (see Definitions \ref{hankeldef} and \ref{markov}). Similarly, 
we alternatively construct $\mathcal{ {\hat H}}_{0,k}$,  each block matrix of which comes from a solution $\hat G_k$ to the $\ell_1$-norm estimator \eqref{l1est}. 
    For the estimate of $D$, we denote $\hat D^{(k)}$ as the first $r\times m$ submatrix of $\hat G_k$. Now, recall the  balanced truncated model from Definition \ref{dorder} and let 
   $ \hat U_k, \hat \Sigma_k, \hat V_k $ be the singular value decomposition (SVD) of the zero-padded matrix $(\mathcal{\bar {\hat H}}_{0,k})_{[1:rk+r], [1:mk]}$. Then, the estimates for $(A,B,C)$ are derived as 
   \begin{align*}
        &\hat C^{(k)} = (\hat U_k\hat \Sigma_k^{1/2})_{[1:r], [1:k]}, \quad \hat B^{(k)} = (\hat \Sigma_k^{1/2} \hat V_k^T)_{[1:k], [1:m]} \\ & \hat A^{(k)} = (\hat U_k\hat \Sigma_k^{1/2})_{[1:rk], [1:k]}^{\dagger} (\hat U_k\hat \Sigma_k^{1/2})_{[r+1:rk+r], [1:k]}.
    \end{align*}
    Note that we have truncated $\mathcal{\bar {\hat H}}_{0,k}$ up to $rk+r$ rows and $mk$ columns, since all milder truncations also yield the same mathematical result. However, fewer rows/columns can be truncated for the sake of numerical stability. 
\end{definition}

We leverage the following lemma that bounds the estimation error of $A, B, C$ from that of the full Hankel matrix (see Proposition 14.2 in \cite{sarkar2021finite}). 

\begin{lemma}\label{sarkar1}
    For $d\in\{1,\dots,k\}$, 
    consider a positive constant $\epsilon_d$ such that
    $\|\mathcal{H}_{0,\infty} -\mathcal{\bar {\hat H}}_{0,d} \|_2 \leq \epsilon_d$. Then, there exists an orthogonal matrix $Q_d \in \mathbb{R}^{n\times n}$ such that 
    \begin{align*}
        &\max\{\|C^{(d)}-\hat C^{(d)} Q_d\|_2, \|B^{(d)}-Q_d^{-1}\hat B^{(d)}\|_2\} \leq \Theta\Bigr(\frac{d\epsilon_d}{\sqrt{\hat \sigma_d}}\Bigr), \\& \|A^{(d)}-Q_d^{-1} \hat A^{(d)} Q_d\|_2\leq \Theta\Bigr(\frac{d\epsilon_d\cdot \|A\|_2}{\hat \sigma_d}\Bigr),
    \end{align*}
    where $\hat\sigma_d$ denotes the $d^\text{th}$ largest singular value of $(\mathcal{\bar {\hat H}}_{0,k})_{[1:rd+r], [1:md]}$.
\end{lemma}

\begin{theorem}\label{retrieve}
    Suppose that Assumptions \ref{stability}, \ref{maxnorm}, and \ref{null} hold. Let $u_t\sim N(0, \gamma^2 I_m)$ for all $t$. Define $q:=1-(1-p)^{2k-1} < 0.5$.  
    Given $\delta\in(0,1]$, after the finite time in \eqref{realtimebound}, 
    there exists an orthogonal matrix $Q_k\in\mathbb{R}^{n\times n}$  such that 
    \begin{align*}
        &\|D-\hat{D}^{(k)}\|_F \leq \Theta\biggr(\frac{\sqrt{r}\|A^{2k-1}\|_2 \cdot \nu}{1-\|A\|_2}\biggr),\\[8pt]
        & \max\{\|C^{(k)}-\hat C^{(k)} Q_k\|_2, \|B^{(k)}-Q_k^{-1}\hat B^{(k)}\|_2\} \\&\hspace{15mm}\leq \Theta\biggr(\frac{\max\{\|A^{k}\|_2,\sqrt{kr}\|A^{2k-1}\|_2 \}\cdot k\nu }{\sqrt{\hat \sigma_k}(1-\|A\|_2)}\biggr), \\[8pt]
        & \|A^{(k)}-Q_k^{-1} \hat A^{(k)} Q_k\|_2   
    \\&\hspace{15mm}\leq \Theta\biggr(\frac{\max\{\|A^{k}\|_2,\sqrt{kr}\|A^{2k-1}\|_2 \}\cdot k\nu }{\hat \sigma_k(1-\|A\|_2)/\|A\|_2}\biggr)
    \end{align*}
     with probability at least $1-\delta$, where $\hat\sigma_k$ denotes the $k^\text{th}$ largest singular value of $(\mathcal{\bar {\hat H}}_{0,k})_{[1:rk+r], [1:mk]}$ and $\nu = \frac{\|C\|_2}{1-2q}\cdot \Bigr(\frac{\eta}{\gamma}+\sqrt{m}\|B\|_2\Bigr).$ 
\end{theorem}

\begin{proof}
By Theorem \ref{esterror}, we directly have $\|G_k^* - \hat{G}_k\|_F$. 
Since $D$ and $\hat D^{(k)}$ are the first $r\times m$ submatrix of $G_k^*$ and $\hat G_k$ respectively, we have $\|D-\hat D^{(k)}\|_F \leq \|G_k^* - \hat G_k\|_F$. 

For estimation errors for $\hat A^{(k)}, \hat B^{(k)}, \hat C^{(k)}$, observe that 
\begin{equation}\label{tri}
    \|\mathcal{H}_{0,\infty} -\mathcal{\bar {\hat H}}_{0,k} \|_2\leq \|\mathcal{H}_{0,\infty} -\mathcal{\bar { H}}_{0,k} \|_2+\|  \mathcal{\bar H}_{0,k} - \mathcal{\bar {\hat H}}_{0,k}\|_2.
\end{equation}
For the first term, note that 
\begin{equation*}
\|\mathcal{H}_{0,\infty} -\mathcal{\bar { H}}_{0,k} \|_2=\begin{bmatrix}
    0 & H_{12}\\ H_{21}&H_{22}
\end{bmatrix},
\end{equation*}
where $\begin{bmatrix}
    H_{21}&H_{22}
\end{bmatrix}=\begin{bmatrix}
    H_{12}\\H_{22}
\end{bmatrix}=\mathcal{H}_{k,\infty}$.
Considering that the squared spectral norm of a matrix is bounded by the sum of the squared spectral norms of its submatrices, we have 
\begin{align*}
    \|&\mathcal{H}_{0,\infty} -\mathcal{\bar { H}}_{0,k} \|_2 \leq (\|\begin{bmatrix}
    H_{21}&H_{22}
\end{bmatrix}\|_2^2 + \|H_{12}\|_2^2)^{1/2}\\&\hspace{3mm}\leq 
\sqrt{2} \|\mathcal{H}_{k,\infty}\|_2 = \sqrt{2}\left\|
\begin{bmatrix}
    C \\CA \\\vdots
\end{bmatrix}A^k
\begin{bmatrix}
    B &AB&\cdots
\end{bmatrix}
\right\|_2\\&\hspace{3mm}\leq \sqrt{2} \bigr(\sum_{i=0}^\infty \|CA^{i}\|_2^2\bigr)^{1/2}\cdot \|A^k\|_2 \cdot \bigr(\sum_{i=0}^\infty \|A^{i}B\|_2^2\bigr)^{1/2}\\&\hspace{3mm}\leq \frac{\sqrt{2} \|C\|_2  \|A^k\|_2\|B\|_2 }{1-\|A\|_2^2}< \frac{\sqrt{2} \|C\|_2  \|A^k\|_2\|B\|_2 }{1-\|A\|_2}.\numberthis\label{firstbound}
\end{align*}
For the second term, note that each block matrix consisting of rows $(i-1)r+1$ to $ir$ of $  \mathcal{\bar H}_{0,k} - \mathcal{\bar {\hat H}}_{0,k}$ for $i=1, \dots, k$ is a submatrix of $G_k^* - \hat G_k$ by the construction of Hankel matrices. 
Thus, we have
\begin{equation}\label{secondbound}
    \|\mathcal{\bar H}_{0,k} - \mathcal{\bar {\hat H}}_{0,k}\|_2 \leq \sqrt{k} \|G_k^* - \hat G_k\|_2\leq \sqrt{k} \|G_k^* - \hat G_k\|_F.
\end{equation}
Substituting \eqref{firstbound} and \eqref{secondbound} into \eqref{tri} yields the bound 
\begin{align*}
    \|\mathcal{H}_{0,\infty} -\mathcal{\bar {\hat H}}_{0,k} \|_2\leq \Theta\biggr(\frac{\max\{\|A^{k}\|_2,\sqrt{kr}\|A^{2k-1}\|_2 \} \cdot \nu}{1-\|A\|_2}\biggr),
\end{align*} 
where $\nu$ is the constant specified in the theorem. This result is then followed by Lemma \ref{sarkar1} to complete the proof.
\end{proof}

\begin{remark}\label{remark4}
In Theorem \ref{retrieve}, we established the estimation error of the $k$-order balanced truncation model. In light of Lemma \ref{sarkar1}, it is possible to retrieve the estimates of all $d$-order balanced models (see Definition \ref{dorder}) for any $d\in\{1,\dots, k\}. $ Specifically, by replacing $k$ with any $d$ in all the steps throughout the theorems, the error bound can be modified accordingly to reflect $d$ instead of $k$.
    The estimation error for the Hankel matrix in the theorem turns out to be $\Theta(\max\{\|A^d\|_2, d\|A^{2d-1}\|_2\})$, which is $\Theta(\|A^d\|_2)$ for a sufficiently large $d$. 
    This implies that as $d$ grows, the estimation error may not initially decrease for small $d$ but eventually experiences an exponential decay. 
\end{remark}
\section{Numerical Experiments} \label{sec: experiments}

To effectively demonstrate the results of this paper, we will provide two examples in this section. 

\begin{example}\label{support2}
In this example, we illustrate the results in Section \ref{exact recovery}, showing that the $\ell_1$-norm estimator indeed recovers the Markov parameter matrix unlike the classical least-squares method in the presence of arbitrary attacks. We 
use $n=300$, $m=6$, $r=9$, and $k=5~\text{or}~10$. We generate two different matrices: a nilpotent $A$ is constructed by assigning $i$ to the $i^\text{th}$ superdiagonal entry, while every $(2k-1)^\text{th}$ superdiagonal entry and all other entries are zero, and a general $A$ is constructed by selecting all entries from Uniform$[-1,1]$. These matrices are then scaled to satisfy $\|A\|_2 = 0.6$. 
The initial state is set to a vector of $1000$s, the control inputs at each time is designed to follow $N(0, 100I_{m})$, and the attack time probability is set to $p=\frac{1}{4k}$ to satisfy  Assumption \ref{null}. The attack $w_t$ is Gaussian with covariance $25I_{n}$ and a mean vector whose entries are either $300$ or $1000$, depending on the sign of the corresponding coordinate of $x_t$. Figure \ref{ex2} shows the estimation error on a log scale over time, where the least-squares method fails to recover the Markov parameter matrix, resulting in an error of at least $10^3$. In contrast, the $\ell_1$-norm estimator yields an  error of zero for the nilpotent $A$ for both $k=5$ and $10$. For the general $A$, one can observe that a larger $k$ results in an error approaching that of the nilpotent case, although a longer time is required for the convergence.
  This strongly supports Theorem \ref{esterror} and the corresponding required time given in \eqref{timebound}.
\end{example}

\begin{figure}[t]
     \centering
     \begin{subfigure}[b]{0.235\textwidth}
         \centering
    \includegraphics[width=\textwidth]{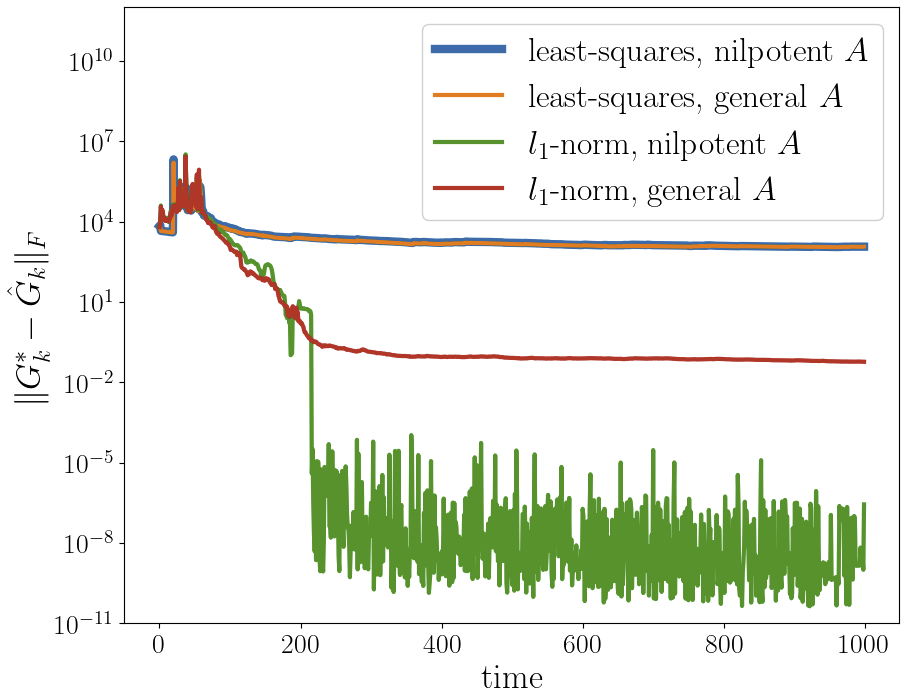}
         \caption{$k=5$}
         \label{p45}
     \end{subfigure}
     \begin{subfigure}[b]{0.235\textwidth}
         \centering
    \includegraphics[width=\textwidth]{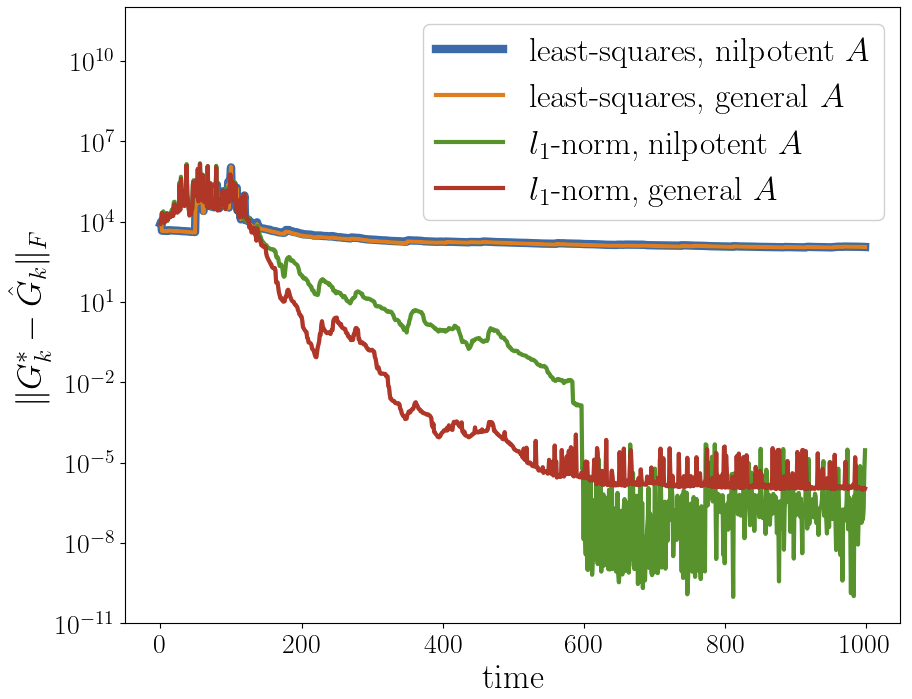}
         \caption{$k=10$}
         \label{p47}
     \end{subfigure}
        \caption{Estimation error for the Markov parameter matrix: $\ell_1$-norm estimator vs. least-squares under adversarial attacks.}
        \label{ex2}
\end{figure}
\begin{figure}[t]
     \centering
     \begin{subfigure}[b]{0.235\textwidth}
         \centering
    \includegraphics[width=\textwidth]{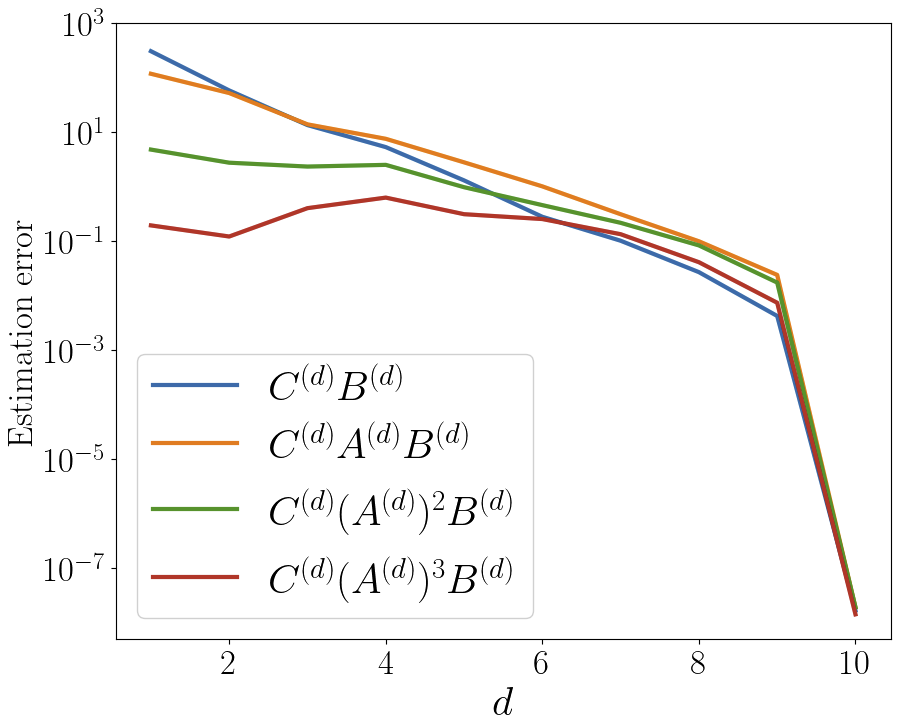}
         \caption{Nilpotent system}
         \label{nil}
     \end{subfigure}
     \begin{subfigure}[b]{0.235\textwidth}
         \centering
    \includegraphics[width=\textwidth]{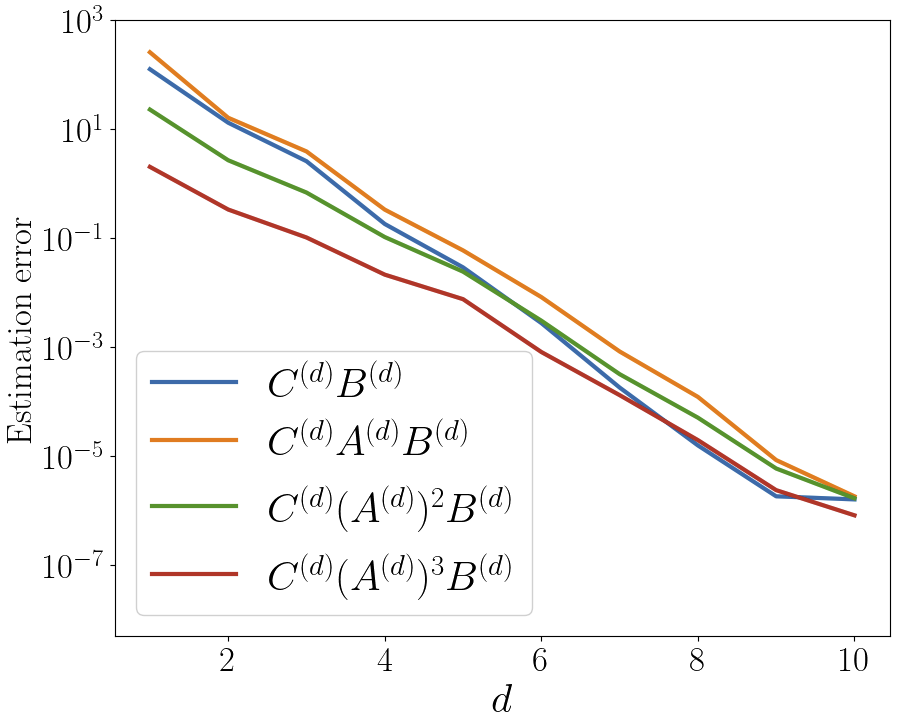}
         \caption{General system}
         \label{gen}
     \end{subfigure}
        \caption{Estimation error for the $d$-order balanced truncated model for $d\in\{1,\dots,k\}$ under adversarial attacks.}
        \label{ex3}
\end{figure}

\begin{example}\label{support3}
 In this example, we demonstrate the recovery of $d$-order balanced truncated models for $d\in\{1,\dots,k\}$ to illustrate the results of Section \ref{retrievesys}.  Due to the existence of infinitely many systems within similarity transformation, we cannot verify whether the estimations of $A^{(d)},B^{(d)},C^{(d)}$ match the true balanced truncation matrices. Thus, we first retrieve  $\hat  A^{(d)} ,\hat  B^{(d)}, \hat  C^{(d)}$ from a reasonable estimate of the Markov parameter matrix at a fixed time, followed by computing $\|C^{(d)}(A^{(d)})^{i}B^{(d)}-\hat C^{(d)} (\hat{A}^{(d)})^i \hat B^{(d)}\|_2$. Figure \ref{ex3} shows this estimation error on a logarithmic scale for $i=0,1,2,3$ and $d\in\{1,\dots, 10\}$, where we adopt the same setting as in Example \ref{support2} with $k=10$. We select time $700$ since Figure \ref{p47} indicates that the estimation error for the Markov parameter matrix has stabilized by this point. One can observe that the nilpotent system naturally shows a lower estimation error than the general system. More importantly, both systems validate the expositions in Remark \ref{remark4}, showing that the estimation error decays exponentially as $d$ increases. 
\end{example}

\section{Conclusion} \label{sec: conclusion}

In this paper, we design the $\ell_1$-norm estimator in terms of control inputs and observations to estimate the Markov parameter matrix. With the goal of obtaining balanced truncated models of the system up to order $k$, we prove that the estimation error is exactly zero for nilpotent systems
and decays exponentially in $k$ for general systems when $p<\frac{1}{4k-2}$. This exponentially decaying error carries over to the estimation error for the balanced truncations of the true system.
This work presents the first result in the literature demonstrating the possibility of accurately learning systems from partial observations under adversarial attacks.

\appendix

\subsection{Proof of Lemma \ref{lemmafixed}}\label{lemma3proof}

\begin{proof}
Let $f(\mathbf{U}) = \sum_{t=2k-1}^{T+2k-2} \mathbb{I}\{v_t^i = 0\}\cdot |s^T \mathbf{U}_t^{(k)}| $, where $\mathbf{U} = [u_0 ~ u_1 ~\cdots u_{T+2k-2}]$. Similarly, let $\mathbf{U}' = [u_0' ~ u_1' ~\cdots u_{T+2k-2}']$. For $\mathbf{U}, \mathbf{U}'$, we have 
\begin{align*}
    &\hspace{-2mm}|f(\mathbf{U}) - f(\mathbf{U}')| \\&\hspace{-2mm}\leq \sum_{t=2k-1}^{T+2k-2}
        \| s^T[(u_t-u_t') ~ \dots ~ (u_{t-2k+1}-u_{t-2k+1}')]\|_2 \\ 
        &\hspace{-2mm}\leq  \sqrt{T} \sqrt{\sum_{t=2k-1}^{T+2k-2}
        \|[(u_t-u_t') ~ \dots ~ (u_{t-2k+1}-u_{t-2k+1}')]\|_2^2 }\\&\hspace{-2mm}\leq  \sqrt{T} \sqrt{2k}\sqrt{\sum_{t=0}^{T+2k-2}\|u_t - u_t'\|_2^2} \leq  \sqrt{2Tk} \|\mathbf{U} - \mathbf{U}'\|_2 \numberthis\label{lipschitzbound}
\end{align*}
due to the Cauchy-Schwarz inequality. This implies that $f$ has a Lipschitz constant of $\sqrt{2Tk}$. 
By the Gaussian concentration inequality (Theorem 5.2.3. in \cite{vershynin2025high}), we have $\|f(\mathbf{U}) - \mathbb{E}[f(\mathbf{U})] \|_{\psi_2} \leq \Theta(\gamma L)$, where $L$ denotes the Lipschitz constant.  Applying this inequality with the sub-Gaussian property \eqref{minus}, we arrive at 
\begin{align*}
    &\mathbb{P}\Biggr(\sum_{t=2k-1}^{T+2k-2} \mathbb{I}\{v_t^i = 0\}\cdot |s^T \mathbf{U}_t^{(k)}| -\mathbb{E}\left[\mathbb{I}\{v_t^i = 0\}\cdot |s^T \mathbf{U}_t^{(k)}|\right] \\&\hspace{10mm}\geq -\frac{\gamma (1-2q)T}{\sqrt{2\pi}} \Biggr) \geq 1-\exp\Bigr(-\Theta\Bigr(\frac{T(1-2q)^2}{k}\Bigr)\Bigr).
\end{align*}
Now, it suffices to show that \[
\mathbb{E}\left[\sum_{t=2k-1}^{T+2k-2}\mathbb{I}\{v_t^i = 0\}\cdot |s^T \mathbf{U}_t^{(k)}|\right] \geq \gamma (1-2q)T\sqrt{\frac{2}{\pi}}.\]
Note that for all $i$, the event $\{v_t^i=0\}$ includes the event $\{w_{t-1}= 0,\cdots, w_{t-2k+1}= 0\}$ (see \eqref{secondline}). Thus, we have
\begin{align*}
    &\mathbb{E}\left[\sum_{t=2k-1}^{T+2k-2}\mathbb{I}\{v_t^i = 0\}\cdot |s^T \mathbf{U}_t^{(k)}|\right]\\&\hspace{5mm}\geq T\cdot \mathbb{E}\Bigr[\mathbb{I}\{w_{t-1}= 0 ,\cdots, w_{t-2k+1}= 0\}\cdot |s^T \mathbf{U}_t^{(k)}|\Bigr]\\&\hspace{5mm}\geq T\cdot \mathbb{E}\left[\mathbb{I}\{\xi_{t-1}= 0 ,\cdots, \xi_{t-2k+1}= 0\}\right]\cdot \mathbb{E}\bigr[ |s^T \mathbf{U}_t^{(k)}|\bigr]\\&\hspace{5mm}\geq T(1\cdot (1-q)+(-1)\cdot q)\cdot \gamma \sqrt{\frac{2}{\pi}}=T(1-2q)\gamma \sqrt{\frac{2}{\pi}},
\end{align*}
where the second inequality follows from \eqref{include} and the independence between $\mathbf{U}_t^{(k)}$ and  $\xi_{t-1}, \dots \xi_{t-2k+1}$; the last uses 
$q:=1-(1-p)^{2k-1}$ being the probability that not all $\xi_{t-1}, \dots \xi_{t-2k+1}$ are zero and $s^T \mathbf{U}_t^{(k)}\sim N(0,\gamma^2)$.
\end{proof}

\subsection{Proof of Lemma \ref{lemmadiff}}\label{lemma4proof}

\begin{proof}
     Let $s, \tilde s\in\mathbb{S}^{2km-1}$. Then, we have 
 \begin{align*}
     &\sum_{t=2k-1}^{T+2k-2}\mathbb{I}\{v_t^i = 0\}\cdot |s^T \mathbf{U}_t^{(k)}| - \sum_{t=2k-1}^{T+2k-2}\mathbb{I}\{v_t^i = 0\}\cdot |\tilde s^T \mathbf{U}_t^{(k)}| \\&\geq - \sum_{t=2k-1}^{T+2k-2} |(s-\tilde s)^T \mathbf{U}_t^{(k)}| \geq -\|s-\tilde s\|_2 \sum_{t=2k-1}^{T+2k-2}  \|\mathbf{U}_t^{(k)}\|_2 \numberthis\label{ssdiff}
 \end{align*}
The quantity $\sum_{t=2k-1}^{T+2k-2}  \|\mathbf{U}_t^{(k)}\|_2$ can be shown to have a Lipschitz constant of $\sqrt{2Tk}$, following an approach similar to \eqref{lipschitzbound}. Applying the Gaussian concentration inequality with  property \eqref{plus}, one arrives at
 \begin{align*}
     &\mathbb{P}\biggr(\sum_{t=2k-1}^{T+2k-2}  \|\mathbf{U}_t^{(k)}\|_2-\mathbb{E}[ \|\mathbf{U}_t^{(k)}\|_2] \leq T \gamma \sqrt{km}\biggr)\\&\geq 1-\exp\Bigr(-\Theta\Bigr(\frac{T^2 \gamma^2   km}{T \gamma^2 k}\Bigr)\Bigr) = 1-\exp(-\Theta(Tm)).
 \end{align*}
 Using
$\mathbb{E}[\|\mathbf{U}_t^{(k)}\|_2] \leq \Theta(\gamma \sqrt{km})$, the inequality can also be written as
\begin{align*}
    \mathbb{P}\biggr(\sum_{t=2k-1}^{T+2k-2} \|\mathbf{U}_t^{(k)}\|_2 \leq 2\cdot \Theta\bigr(T\gamma \sqrt{km}\bigr)\biggr) \geq 1-\frac{\delta}{2}
\end{align*}
when $T\geq \Theta( \frac{1}{m}\log(\frac{2}{\delta}))$. Considering the lower bound of \eqref{ssdiff} completes the proof.
\end{proof}

\subsection{Proof of Lemma \ref{sumxt}}\label{lemma5proof}
\begin{proof}
    Due to the system dynamics \eqref{sysd}, we have 
    \begin{align*}
        &\sum_{t=0}^{T-1} \|x_t\|_2 = \sum_{t=0}^{T-1} \Bigr\|A^t x_0 + \sum_{i=0}^{t-1} (A^{t-1-i}Bu_i+A^{t-1-i} w_{i})\Bigr\|_2 
        \\&\hspace{16mm}< \sum_{i=0}^{\infty} \|A\|_2^i   \Bigr[\|x_0\|_{2} + \sum_{t=0}^{T-2} (\|w_t\|_{2} +  \|Bu_t\|_{2})\Bigr]\\&\hspace{3mm}\leq\frac{1}{1-\|A\|_2}\Bigr[\|x_0\|_{2} + \sum_{t=0}^{T-2} (\|w_t\|_{2} +  \|B\|_2\|u_t\|_{2})\Bigr] \numberthis\label{submul}
    \end{align*}
    due to the triangle inequality. Let $S_T$ denote the term in \eqref{submul}.
    Under Assumption \ref{maxnorm}, the sub-Gaussian norms of $\|x_0\|_2$ and $\|w_t\|_2$ are bounded by $\eta$. Moreover, by Lemma \ref{independent}, the sub-Gaussian norm of $\|u_t\|_2$ is bounded by $\gamma \sqrt{m}$.
      Using the filtration $\mathcal{F}_t$ and Lemma \ref{centering}, it follows that $\|S_T-\mathbb{E}[S_T]\|_{\psi_2} \leq \frac{\Theta((\eta + \gamma \sqrt{m})\sqrt{T})}{1-\|A\|_2}$.
    Applying  property \eqref{plus}, one arrives at
    \begin{align*}
        \mathbb{P}\Bigr( S_T-\mathbb{E}[S_T] \leq \frac{(\eta+\gamma \sqrt{m}\cdot \|B\|_2)T}{1-\|A\|_2}   \Bigr) \geq 1-\exp(-\Theta(T)).
    \end{align*}
   Note that $\mathbb{E}[S_T]$ can be upper bounded using property \eqref{firstprop}; specifically, $\mathbb{E}[\|u_t\|_2] \leq \Theta(\gamma \sqrt{m})$, $\mathbb{E}[\|x_0\|_2] \leq \eta$, and $\mathbb{E}[\|w_t\|_2] \leq \eta$ given $\mathcal{F}_t$. Since  $\sum_{t=0}^{T-1} \|x_t\|_2\leq S_T$,  this concludes the proof.
\end{proof}

\vspace{3mm}
\printbibliography

@book{ben2009sto,
  title={Stochastic Control of Partially Observable Systems},
  author={Alain Bensoussan},
  year={1992},
  publisher={Cambridge University Press}
}

@book{ljung1998sys,
  title={System Identification: Theory for the User},
  author={Lennart Ljung},
  edition={2},
  year={1998},
  publisher={Pearson}
}

@article{kim2025acc,
  title={Prevailing against Adversarial Noncentral Disturbances: Exact Recovery of Linear Systems with the $l_1$-norm Estimator},
  author={Jihun Kim and Javad Lavaei},
  journal={arXiv preprint arXiv:2410.03218},
  year={2025}
}

@article{zhang2024exact,
  title={Exact Recovery Guarantees for Parameterized Nonlinear System Identification Problem under Sparse Disturbances or Semi-Oblivious Attacks},
  author={Haixiang Zhang and Baturalp Yalcin and Javad Lavaei and Eduardo D. Sontag},
journal={Transactions on Machine Learning Research},
  year={2025},
 publisher={JMLR}
}

@book{vershynin2025high,
  title={High-Dimensional Probability: An Introduction with Applications in Data Science},
  author={Roman Vershynin},
  year={2025},
    edition={2},
  publisher={Cambridge University Press}
}

@article{yalcin2023exact,
  title={Exact Recovery for System Identification with More Corrupt Data than Clean Data},
  author={Baturalp Yalcin and Haixiang Zhang and Javad Lavaei and Murat Arcak},
  journal={IEEE Open Journal of Control Systems},
    volume={4},
pages={1--17},
  year={2025},
publisher={IEEE}
}

@inproceedings{fledder2002markov,
  title={A Comparison of Least Squares Algorithms for Estimating Markov Parameters},
  author={Matthew S. Fledderjohn and Matthew S. Holzel and Harish J. Palanthandalam-Madapusi and Robert J. Fuentes and Dennis S. Bernstein},
  booktitle={American Control Conference (ACC)},
  year={2010},
  organization={IEEE}
}

@article{sarkar2021finite,
  title={Finite Time {LTI} System Identification},
  author={Tuhin Sarkar and Alexander Rakhlin and Munther A. Dahleh},
  journal={Journal of Machine Learning Research},
  volume={22},
  number={26},
  pages={1--61},
  year={2021}
}

@inproceedings{simchowitz2019semi,
  title={Learning Linear Dynamical Systems with Semi-Parametric Least Squares},
  author={Max Simchowitz and Ross Boczar and Benjamin Recht},
  booktitle={Conference on Learning Theory},
  volume={99},
  pages={1--89},
  year={2019},
  organization={PMLR}
}

@article{hokalman1966kalman,
  title={Effective construction of linear state-variable models from input/output functions},
  author={Bin-Lun Ho and Rudolf E. Kalman},
  journal={Automatisierungstechnik},
  volume={14},
  number={112},
  pages={545--548},
  year={1966},
  publisher={De Gruyter}
}

@inproceedings{jedra2020finite,
  title={Finite-time Identification of Stable Linear Systems Optimality of the Least-Squares Estimator},
  author={Jedra, Yassir and Proutiere, Alexandre},
  booktitle={Conference on Decision and Control},
  year={2020},
  organization={IEEE}
}

@inproceedings{show2016system,
  title={System identification in the presence of adversarial outputs},
  author={Mehrdad Showkatbakhsh and Paulo Tabuada and Suhas Diggavi},
  booktitle={Conference on Decision and Control},
  year={2016},
  organization={IEEE}
}

@inproceedings{skelton1994lqg,
  title={The Data-Based {LQG} Control Problem},
  author={Robert E. Skelton and Guojun Shi},
  booktitle={Conference on Decision and Control},
  year={1994},
  organization={IEEE}
}

@inproceedings{simchowitz2018learningwithout,
  title={Learning Without Mixing: Towards A Sharp Analysis of Linear System Identification},
  author={Max Simchowitz and Horia Mania and Stephen Tu and Michael I. Jordan and Benjamin Recht},
  booktitle={Conference On Learning Theory},
  pages={439--473},
  year={2018}
}

@article{ala2014heart,
  title={Optimizing Cancer Screening Using Partially Observable Markov Decision Processes},
  author={Oguzhan Alagoz},
  journal={INFORMS Tutorials in Operations Research},
  pages={75--89},
  year={2014},
  publisher={INFORMS}
}

@article{kung1981hankel,
  title={Optimal {H}ankel-norm model reductions: Multivariable systems},
  author={Sun-Yuan Kung and David W. Lin},
  journal={IEEE Transactions on Automatic Control},
  volume={26},
  number={4},
  pages={832--852},
  year={1981},
  publisher={IEEE}
}

@article{lauri2022survey,
  title={Partially Observable Markov Decision Processes in Robotics: A Survey},
  author={Mikko Lauri and David Hsu and Joni Pajarinen},
  journal={IEEE Transactions on Robotics},
  volume={39},
  number={1},
  pages={21--40},
  year={2023},
  publisher={IEEE}
}

@article{oymak2022revisit,
  title={Revisiting {H}o–{K}alman-Based System Identification: Robustness and Finite-Sample Analysis},
  author={Samet Oymak and Necmiye Ozay},
  journal={IEEE Transactions on Automatic Control},
  volume={67},
  number={4},
  pages={1914--1928},
  year={2022},
  publisher={IEEE}
}

@article{pasqualetti2013cyber,
  title={Attack Detection and Identification in Cyber-Physical Systems},
  author={Fabio Pasqualetti and Florian D\"{o}rfler and Francesco Bullo},
  journal={IEEE Transactions on Automatic Control},
  volume={58},
  number={11},
  pages={2715--2729},
  year={2013},
  publisher={IEEE}
}

@article{mo2015secure,
  title={Secure Estimation in the Presence of Integrity Attacks},
  author={Yilin Mo and Bruno Sinopoli},
  journal={IEEE Transactions on Automatic Control},
  volume={60},
  number={4},
  pages={1145--1151},
  year={2015},
  publisher={IEEE}
}

@article{wang2016power,
  title={Research on Resilience of Power Systems Under Natural Disasters—A Review},
  author={Yezhou Wang and Chen Chen and Jianhui Wang and Ross Baldick},
  journal={IEEE Transactions on Power Systems},
  volume={31},
  number={2},
  pages={1604--1613},
  year={2016},
  publisher={IEEE}
}

@article{waseem2020partial,
  title={Electricity grid resilience amid various natural disasters: Challenges and solutions},
  author={Muhammad Waseem and Saeed D. Manshadi},
  journal={The Electricity Journal},
  volume={33},
    number = {10},
  pages={106864},
  year={2020},
  publisher={Elsevier}
}

@article{hauser2017cyber,
  title={Cyber attack models for smart grid environments},
  author={Peter Eder-Neuhauser and Tanja Zseby and Joachim Fabini and Gernot Vormayr},
  journal={Sustainable Energy, Grids and Networks},
  volume={12},
  pages={10--29},
  year={2017},
  publisher={Elsevier}
}

@article{fara2018unstable,
  title={Finite Time Identification in Unstable Linear Systems},
  author={Mohamad Kazem Shirani Faradonbeh and Ambuj Tewari and
George Michailidis},
  journal={Automatica},
  volume={96},
  pages={342--353},
  year={2018},
  publisher={Elsevier}
}

\end{document}